\documentclass[12pt]{amsart}
\usepackage{amssymb,amsthm,amsmath,amsfonts}

\usepackage{geometry}
\usepackage{amssymb}
\usepackage{verbatim}
\usepackage{amsmath}
\usepackage[utf8]{inputenc}
\usepackage{enumerate}
\usepackage{verbatim}
\usepackage{amsfonts}
\usepackage{graphicx}
\usepackage{tikz}
\usetikzlibrary{matrix}
\usepackage{lscape}
\usepackage{color}

\newcommand{\E}{\mathbb{{E}}}
\newcommand{\R}{\mathbb{R}}
\newcommand{\Z}{\mathbb{Z}}

\newcommand{\PP}{\mathbb{P}}

\newcommand{\mb}[1]{\mathbb{{#1}}}
\newcommand{\mc}[1]{\mathcal{{#1}}}

\newcommand{\p}[1]{\mb{P} \left( #1 \right)}

\newcommand{\e}{\varepsilon}

\newcommand{\dd}{\mathrm{d}}
\newcommand{\1}{\mathbf{1} }

\DeclareMathOperator{\conv}{conv}

\usepackage{color}

\DeclareMathOperator{\var}{Var}

\DeclareMathOperator{\Pois}{\textbf{Pois}}
\newcommand{\scal}[2]{\left\langle #1, #2\right\rangle}
\theoremstyle{definition}

\newtheorem{theorem}{Theorem}
\newtheorem{lemma}[theorem]{Lemma}

\theoremstyle{remark}
\newtheorem*{remark*}{Remark}

\theoremstyle{definition}

\begin{document}

\newgeometry{tmargin=2.5cm, bmargin=2.5cm, lmargin=2.5cm, rmargin=2.5cm}

\title[]{Log-concavity and discrete degrees of freedom}

\author{Jacek Jakimiuk}
\address{University of Warsaw}
\email{j.jakimiuk4@student.uw.edu.pl}

\author{Daniel Murawski}
\email{dk.murawski@student.uw.edu.pl}

\author{Piotr Nayar}
\thanks{P.N. was supported by the National Science Centre, Poland, grant 2018/31/D/ST1/01355}
\email{nayar@mimuw.edu.pl}

\author{Semen Słobodianiuk}
\email{ss407112@students.mimuw.edu.pl}

\begin{abstract}
We develop the notion of discrete degrees of freedom of a log-concave sequence and use it to prove that geometric distribution minimises R\'enyi entropy of order infinity under fixed variance, among all discrete log-concave random variables in $\Z$.  We also show that the quantity $\mathbb{P}(X=\mathbb{E} X)$ is maximised, among all ultra-log-concave random variables with fixed integral mean, for a Poisson distribution. 
\end{abstract}

\maketitle

{\footnotesize

}
\bigskip

\section{Introduction}\label{sec:intro}

Log-concavity is an important notion  across many disciplines of mathematics, computer science and beyond. In this article we restrict our attention to log-concave functions on the real line $\R$ or on the set of integers $\Z$. Let $A \subseteq \R$ be an interval. A nonnegative function $f:A \to \R$ is called log-concave if $f=e^{-V}$ with $V:A \to \R \cup \{+\infty\}$ convex. The set of such functions will be denoted by $\mc{L}$. Many sharp inequalities for log-concave functions can be formulated as the following type of optimisation problems: find the maximum of a convex functional $\Phi:\mc{L}\to \R$ under $n$ linear constraints $\Phi_i(f)=m_i$, $i=1,\ldots, n$, where $\Phi_i(f)= \int f g_i$ for some function $g_i$. Usually the functionals $\Phi_i$ correspond to various moment constraints, namely $g(x)=|x|^p$ gives the so-called $p$th moment $\int |x|^p f(x) \dd x$. If one wants to consider probability densities $f$ with fixed variance, one needs to introduce three linear constraints given by $g_1 \equiv 1$ (fixing $f$ to be a probability density), $g_2(x)=x$ (fixing the mean $\int x f(x) \dd x$) and $g_3(x)=x^2$ (fixing the second moment $\int x^2 f(x) \dd x$). When more than one or two constraints are present, the above optimisation problem is usually very challenging, one of the reasons being the fact that $\mc{L}$ is not closed under summation. To overcome this difficulty Fradelizi and Gu\'edon introduced in \cite{FG06} a powerful notion of degrees of freedom. The function $f \in \mc{L}$ has $d$ degrees of freedom if there exist linearly independent functions $q_1, \ldots, q_d$ and $\e>0$ such that $f_\delta = f+ \delta_1 q_1+ \ldots + \delta_d q_d$ is log-concave for all $\delta_1, \ldots, \delta_d \in (-\e,\e)$. Intuitively,  $f$ belongs to an affine subspace of dimension $d$ which locally around $f$ is contained in $\mc{L}$. The set of all $f_\delta$ will be called the \emph{cube} around $f$.  Suppose in our optimisation problem we are given $n$ constraints. Then if $d>n$, by a simple linear algebra argument one can find within the cube a linear subspace $\mc{C}$ of dimension $d-n$ such that all functions in this subspace satisfy the constraints. Since this is a linear subspace, one can find $f_\delta, f_{-\delta} \in \mc{C}\cap\mc{L}$. But $f= \frac12(f_\delta+ f_{-\delta})$ and thus $f$ is not an extreme point of $\mc{L}'$, the subset of functions in $\mc{L}$ satisfying the constraints. On the other hand, due to Krein-Milman theorem it is known that the supremum of $\Phi$ on $\mc{L}'$ is attained on some extremal point. It is therefore enough to consider the problem only for function having at most $n$ degrees of freedom. Fradelizi and Gu\'edon showed that such functions must be of the form $f=e^{-V}$ with $V$ being the maximum of $n-1$ affine functions. For such simple functions one can hope to solve the problem by an exact computation. This approach was used in \cite{MNT21, BN21} to lower bound entropy and R\'enyi entropy of log-concave random variables in terms of variance.  It was also used in \cite{M23} by the second named author in  the context of moment comparison.           

Let us now deal with log-concave sequences. We fix a nonnegative integer $n$ and let $[n]=\{0,1,\ldots, n\}$. We say that a sequence $(a(k))_{k \in [n]}$ of nonnegative real numbers is log-concave if $\{k \geq 0: a(k)>0\}$ is a discrete interval and the inequality $a(k)^2 \geq a(k+1) a(k-1)$ is satisfied for $k=1, \ldots, n-1$. This definition can clearly be extended to log-concave sequences on $\Z$. For a random variable $X$ with probability mass function $p$ the mean functional is defined via $p \to \sum_k k p(k)$ and the $p$th moment functional is  $p \to \sum_k |k|^p p(k)$. The R\'enyi entropy functional of order $\alpha \ne 1$ is $H_\alpha(p)=\frac{1}{1-\alpha} \ln\left(\sum_k p(k)^\alpha \right)$, whereas for $\alpha=1$ one recovers in the limit the Shannon entropy $H(p)=-\sum_k p(k) \ln p(k)$,  with the convention $0 \log 0 =0$. One also defines $H_\infty(p)=-\ln M(p)$, where $M(p)=\max_z p(z)$. A random variable $X$ is log-concave if its probability mass function $p$ is log-concave. 

In this article we develop a discrete version of the concept of degrees of freedom and give two applications. In a recent work by Aravinda \cite{A23}, published after the first version of the present article appeared online, the author developed the same concept of degrees of freedom with slightly different proof (in \cite{A23} the continuous case is used as a lemma, whereas here we provide a direct argument). Aravinda then uses this tool to prove that for a log-concave random variable $X$ one has $M(X)^2 (1+\var(X)) \leq 1$. The equality is achieved asymptotically for a geometric distribution $p(k)=\theta(1-\theta)^k \1_{\{0,1,\ldots\}}(k)$  when $\theta \to 0$.  Here we  prove a stronger version of this bound for which equality occurs for every geometric distribution.

\begin{theorem}\label{thm:ent-var}
Let $X$ be a log-concave random variable in $\Z$. Then 
\[
	M(X)^2 \var(X) + M(X) \leq 1
\]
\end{theorem}       

Let us also mention that in \cite{BMM21} the authors proved than any discrete log-concave distribution satisfies the inequality
\[
	\frac{1}{\sqrt{1+ 12 \var(X)}} \leq M(X) \leq \frac{1}{\sqrt{\frac14 + \var(X)}}. 
\] 
The result of Aravinda replaces the constant $\frac14$ with $1$. The left inequality becomes equality for  uniform distributions on  discrete intervals and is therefore optimal only for a countable set of values of $\var(X)$. In \cite{BMM21} it is also proved that in the class of symmetric log-concave distributions one has $M(X)^2 (1+2 \var(X)) \leq 1$, which is optimal for all values of $\var(X)$ as the equality occurs for two-sided geometric distributions. The  inequality from Theorem \ref{thm:ent-var} enjoys the same feature of being optimal for all values of $\var(X)$.    

Our second application concerns the so-called ultra-log-concave discrete random variables. A random variable $X$ taking values in the set of nonnegative integers is called ultra-log-concave if  $\mu(n)/\Pi_\lambda(n)$ is log-concave, where $\Pi_\lambda(n)= e^{-\lambda} \frac{\lambda^n}{n!}$ stands for the probability mass function of the Poisson random variable $\Pois(\lambda)$ with parameter $\lambda>0$. Ultra-log-concave random variables attracted considerable attention of researchers over last two decades. The definition itself is, according to our best knowledge, due to Pemantle \cite{P00} who introduced it in the context of  theory of negative dependence of random variables. Walkup's theorem (see Theorem 1 in \cite{W76}) implies that convolution of two independent ultra-log-concave random variables is ultra-log-concave, see also  \cite{L97, G09, NO12, MM21} for generalisations and different proofs.

More recently, Johnson in \cite{J07} considered ultra-log-concave random variables in the context of Shannon entropy $H(X)$.  He proved that the Poisson distribution maximises entropy in the class of ultra log–concave distributions under fixed mean. The author uses an interesting semigroup technique based on  adding Poisson random variable and the operation of \emph{thinning}. Simple proof of this result was given by Yu in \cite{Y09}. The author proved that if $X$ is ultra-log-concave and $Z \sim \Pois(\E X)$ so that $\E X = \E Z$, then $\E \phi(X) \leq \E \phi(Z)$ for any convex function $\phi:\R \to \R$. Johnson's result follows from this statement easily by using Gibbs' inequality. In \cite{Y09} a generalization of the above results to the so-called compound  distributions is also discussed, providing generalizations and elementary proofs of the results from \cite{JKM13}.  The following theorem goes beyond the convex case discussed by Yu in \cite{Y09}. 

\begin{theorem}\label{thm:poiss}
Let $X$ be an ultra-log-concave random variable with integral mean. Then
\[
	\p{X=\E X} \geq  \p{\Pois(\E X) = \E X}.
\]
\end{theorem}
\noindent This theorem is motivated by Theorem 1.1 from \cite{AMM21}, where the concentration inequalities for ultra-log-concave random variables were derived. 

This article is organised as follows. In Section  \ref{sec:2} we develop the notion of discrete degrees of freedom. Section \ref{sec:3} is devoted to the proof of Theorem \ref{thm:ent-var}. Finally, in Section \ref{sec:4} we give a proof of Theorem \ref{thm:poiss}.

\section{Discrete degrees of freedom}\label{sec:2}

Suppose $p$ is a log-concave sequence supported in some finite discrete interval which without loss of generality can be assumed to be $[L]=\{0,1,\ldots, L\}$. We say that $p$ has $d$ \emph{degrees of freedom}  if there exist linearly independent sequences $q_1, \ldots, q_d$ supported in $[0,L]$ and $\e>0$ such that for all $\delta_1,\ldots, \delta_d \in (-\e,\e)$ the sequence
\[
	p+ \delta_1 q_1 + \ldots + \delta_d q_d
\] 
is log-concave  in $[0,L]$. 

We shall prove the following lemma describing sequences with small number of degrees of freedom. The proof is a rather straightforward adaptation of the argument presented in \cite{FG06}.

\begin{lemma}\label{lem:ddf}
Let $d \geq 1$. Suppose a positive log-concave sequence supported in $[L]$ has $d+1$ degrees of freedom. Then $p=e^{-V}$ with $V=\max(l_1, \ldots, l_d)$, where $l_1,\ldots, l_d$ are arithmetic progressions.
\end{lemma}

\begin{proof}
Since $p$ is strictly positive and log-concave, it can be written in the form $p=e^{-V}$, where $V$ is convex. The  sequence $V'(n)=V(n+1)-V(n)$ is called the \emph{slope sequence}. Clearly the slope sequence is non-decreasing. We prove the lemma by contrapositive. We shall assume that $V$ cannot be written as a maximum of $d$ arithmetic progressions. Our goal is then to prove that $p$ has at least $d+2$ degrees of freedom. 

Define the sequence $n_0,\ldots, n_k$ inductively  by taking $n_0=0$ and $n_{i+1}=\min \{n > n_i: V'(n)>V'(n_i)\}$ as long as the set is non-empty. Thus $V'(n_0)< V'(n_1)< \ldots < V'(n_k)$ with $k \geq d$, as $V$ is not piecewise linear with at most $d$ pieces.


For $i=0,1,\ldots,k$ let us define the sequence $V_i$ via the expression
\[
	V_i(n) = \left\{ \begin{array}{ll}
	V(n) & n \in [0,n_i] \\
	V(n_i)+ V'(n_i)(n-n_i) & n \in [n_i,L]
	\end{array} \right. .
\]
It is not hard to show that $V_i$ are convex. We shall assume that $V'(n_0) \ne 0$ so that the sequence $V_0$ is not constant. If this is not the case it suffices to reflect the picture and use $V'(n_k)$ instead of $V'(n_0)$.  \\

\noindent \textit{Claim 1.} There exists $\e>0$ such that for all $\delta, \delta_0,\ldots, \delta_k \in (-\e,\e)$ the sequence
\[
	e^{-V}( 1+ \delta+ \delta_0 V_0 + \delta_1 V_{1} + \ldots + \delta_k V_{k})
\]  
is log-concave. 

\begin{proof}[Proof of Claim 1]
On each of the intervals $[n_i, n_{i+1}]$, $i=0,\ldots, k+1$, where we take $n_{k+1}=L$, the above sequence is given by the expression of the form $p(n) = e^{-V}(1+\mu_1 W+\mu_2 V)$, where $W$ is an arithmetic progression. We first check that for $n \in (n_i,n_{i+1})$ we have $p(n)^2 \geq p(n+1)p(n-1)$ for $\mu_1, \mu_2$ sufficiently small. By continuity one can assume that $\mu_2 \ne 0$. We want to prove convexity of 
\begin{align*}
	\tilde{V}(n) & = V(n) - \log (1+\mu_1 W(n)+\mu_2 V(n)) = V(n)-  \log\left( 1 + \mu_2 \left( V(n) + \frac{\mu_1}{\mu_2} W(n) \right) \right) \\
	& = -  \frac{\mu_1}{\mu_2} W(n) + V(n) + \frac{\mu_1}{\mu_2} W(n) -  \log\left( 1 + \mu_2 \left( V(n) + \frac{\mu_1}{\mu_2} W(n) \right) \right) \\
	& = -  \frac{\mu_1}{\mu_2} W(n) + h_{\mu_2}\left(  V(n) + \frac{\mu_1}{\mu_2} W(n) \right),
\end{align*}
where
\[
	h_\mu(t) = t - \log(1+\mu t). 
\]
The first term is affine. For small $\mu$ the function $h_\mu(t)$ is increasing and convex. Since the sequence $V+ \frac{\mu_1}{\mu_2} W$ is convex, it is enough to show that $f(g(n))$ is a convex sequence whenever $f$ is an increasing convex function and $g$ is convex. This is straightforward since
\[
	f(g(n)) \leq  f\left(\frac12g(n+1)+ \frac12 g(n-1) \right) \leq \frac12 f(g(n+1))  + \frac12 f(g(n-1)). 
\] 

Now we are left with checking our inequality in points $n=n_i$. But since then $V(n) < \frac12 V(n+1)+ \frac12 V(n-1)$, the inequality follows by a simple continuity argument. 

\end{proof}


\noindent \textit{Claim 2.} The sequences $1, V_{0}, V_{1}, \ldots, V_{k}$ are linearly independent.

\begin{proof}[Proof of Claim 2]
 Let $V_{-1}=1$. Let us consider $U_{-1} = V_{-1}$, $U_0 = V_0- V_{-1}, \ldots, U_k = V_k- V_{k-1}$. To prove that $V_i$ are linearly independent, it suffices to show that $U_i$ are linearly independent. Indeed, suppose that $\sum_{i=-1}^k b_i V_i \equiv 0$. This means that 
\[
b_{-1} U_{-1} + b_0(U_{-1}+U_0)+ \ldots + b_k(U_{-1}+U_0+\ldots + U_k) \equiv 0,
\]
which is 
\[
	(b_{-1}+\ldots+b_k) U_0 + (b_0+\ldots+b_k) U_1 + \ldots + b_k U_k \equiv 0. 
\]
If $U_i$ are linearly independent, it follows that $b_i+\ldots+b_k = 0$ for $i=-1,0,\ldots, k$, which easily leads to $b_i=0$ for all $i=-1,\ldots,k$. 

Now the fact that $U_i$ are linearly independent is easy since $U_i$ for $i \geq 1$ is supported in $[n_{i}+1,L]$. These intervals form a decreasing sequence, so in order to show that every combination $\sum_{i=-1}^k b_i U_i \equiv 0$ in fact has zero coefficients it is enough to evaluate this equality first at points $n=0,1$ to conclude that $b_{-1}=b_0=0$ (note that the support of $U_i$ for $i \geq 1$ is contained in $[2,L]$) and then consecutively at points $n_{1}+1, n_2+1, \ldots, n_k+1$ to conclude that $b_1= b_2=\ldots = b_k=0$.
\end{proof}

Combining Claim 1 and Claim 2 finishes the proof.

\end{proof}

Let us now consider the space $\mc{L}_I$ of all log-concave sequences $(p_i)_{i \in I}$ is an interval $I$. We shall identify the sequence $p$ with a vector in $\R^{|I|}$. Suppose we are given vectors $v_1, \ldots, v_d \in \R^{|I|}$ and real numbers $a_1,\ldots, a_d$. Let us introduce the polytope 
\[
P_d(a,v)=\{p: \scal{p}{v_i}=a_i, \ i=1,\ldots, d\} \cap [0,\infty)^{|I|}.
\] 
We will be assuming that this polytope is bounded, which will be the case in our applications. Let us now assume that we are given a convex continuous functional $\Phi:\R^{|I|} \to \R$. The following lemma is well known and can be found in the continuous setting in  \cite{FG06}.

\begin{lemma}\label{lem:convex}
The supremum of a convex continuous functional $\Phi:\R^{|I|} \to \R$ on $\mc{L}_I \cap P_d(a,v)$ is attained on some sequence having at most $d$ degrees of freedom.
\end{lemma}

\begin{proof}
Let $A = \mc{L}_I \cap P_d(a,v)$. By compactness of $A$ the supremum $m$ of $\Phi$ on $A$ is attained. By convexity of $\Phi$ on $K=\conv(A)$ the maximum is the same as the maximum on $A$ and is attained in some point $x \in K$. Moreover, as we work in a finite dimensional Euclidean space, $K$ is also compact. A baby version of the Krein-Milman theorem shows that $x$ is a convex combination of extreme points of $K$, that is $x=\sum \lambda_i x_i$, where positive numbers $\lambda_i$ sum up to one. By convexity $\Phi$ attains its maximum on $K$ also in all the points $x_i$. Thus, the maximum of $\Phi$ on $K$ is attained in some extreme point of $K$. Clearly extreme points of $K$ must belong to $A$. It is therefore enough to show that if $p \in K$ has more than $d$ degrees of freedom, then $p$ is not an extreme point of $K$.

Suppose $p \in A$ with support $I' \subseteq I$ and there exist linearly independent sequences $q_1, \ldots, q_k$ supported in $I'$ and $\e>0$ such that for all $\delta_1,\ldots, \delta_k \in (-\e,\e)$ the sequence
\[
	p_\delta = p+ \delta_1 q_1 + \ldots + \delta_k q_k
\] 
is log-concave in $I'$ and thus also in $I$. Therefore, it belongs to $\mc{L}_I$. Note that the set of parameters $\delta=(\delta_1,\ldots, \delta_k)$ for which $p_\delta \in P_d(a,v)$ form a linear subspace of dimension at least $k-d$. If $k \geq d+1$ then this subspace is non-trivial and contains two antipodal points $\delta$ and $-\delta$. Note that $p=\frac12 p_{\delta}+\frac12 p_{-\delta}$ and thus $p$ is not an extreme point of $K$ as both $p_\delta$ and $p_{-\delta}$  belong to $K$.

     
\end{proof}

\section{Proof of Theorem \ref{thm:ent-var}}\label{sec:3}

\noindent \emph{Step 1.} We shall follow the strategy described in the introduction to reduce the inequality to the case of simple functions. By translation invariance and approximation arguments one can assume that $X$ has its  support contained in $[L]$. Let $p$ be the probability mass function of $X$. We shall maximise the convex functional $p \to \max_{z \in [L]} p(z)$ under three linear constraints gives by vectors $v_1=(1, \ldots, 1)$, fixing $p$ to be a probability distribution, $v_2=(0,1,\ldots,L$), fixing $\E X$, and $v_3=(0^2,1^2, \ldots, L^2)$, fixing $\E X^2$. Lemma \ref{lem:convex} thus shows that the maximum is attained for $p$ having at most 3 degrees of freedom, which by Lemma \ref{lem:ddf} implies that $p=e^{-V}$ where $V=\max(l_1,l_2)$, with $l_1,l_2$ being linear progressions. It is therefore enough to check the inequality for sequences of the form
\[	
	p(n)= C e^{-\beta_1(n-N)} \1_{[0,N]}(n) + C e^{-\beta_2(n-N)} \1_{[N+1,L]}(n) , \qquad N \leq L, \ \beta_1 \leq \beta_2.
\]

\noindent \emph{Step 2.} As in \cite{A23} we prove the inequality for $p$ monotone. We sketch the argument for completeness. We can assume that $p$ is non-increasing. In fact we shall not restrict ourselves to the above special functions on $[L]$, but consider general functions on $\Z_+ = \{0,1,\ldots\}$.    Consider geometric distribution $Y$ with parameter $\theta=(1+\E X)^{-1}$ so that $\E X = \E Y$.  The probability mass function of $Y$ will be denoted by $q$. Since $q$ is log-affine and $p$ is log-concave,  $p-q$ changes sign at most two times. In fact it changes sign precisely two times in points $k_1, k_2$, since in case of one sign change point $k_1$ we would get 
\[
0=  \E X - \E Y = \sum_k k (p(k)-q(k)) = \sum_k (k-k_1) (p(k)-q(k)),   
\]
contradiction, as  $(k-k_1) (p(k)-q(k))$ has fixed sign. The sign pattern of $p-q$ is $(-,+,-)$ and therefore 
\[
M(X) = \max_z p(z) = p(0) \leq q(0) = \max_z q(z)=M(Y) .
\]
 We claim that also $\E X^2 \leq \E Y^2$. Choose numbers $a,b$ such that $k^2 - ak-b=0$ for $k=k_1, k_2$. Then using $\E X = \E Y$ we get
\[
	\E X^2 - \E Y^2 = \sum_k k^2(p(k)-q(k)) = \sum_k (k^2-ak-b)(p(k)-q(k)) \leq 0, 
\] 
since the sign pattern of $k^2-ak-b$ is $(+,-,+)$ and thus the expression under the sum is non-positive. We conclude that
\[
	M(X)^2 \var(X)+M(X) \leq M(Y)^2 \var(Y)+M(Y) = 1.
\] 

\begin{remark*}
The fact that under $\E X = \E Y$ the pattern $(-,+,-)$ implies $\E \phi(X) \leq \E \phi(Y)$ for convex $\phi$ seems to be due to Barthe and Naor  \cite{BN02}. The proof technique presented here appeared in \cite{ENT18-2} in the context of moment inequalities. See also \cite{E21} for another application of this technique to moment inequalities and \cite{BNZ21} for an application in the context of Shannon entropy.     
\end{remark*}

\noindent \emph{Step 3.} Let us assume that $p$ is not monotone, in which case $C= \max_{z \in [L]} p(z)$. We now follow the strategy from \cite{A23}. Take $p_1 = e^{\beta_1}$, $p_2 = e^{-\beta_2}$ and $K=L-N$. Define also
\[
	S(x;N) = \sum_{n=1}^N x^n, \qquad S_1(x;N)= \sum_{n=1}^N n x^n \qquad \textrm{and} \qquad S_2(x;N)= \sum_{n=1}^N n^2 x^n.
\]
We shall also use
\[
	Z(x,y; N,K) = 1+S(x,N)+S(y,K).
\]
Then $\sum_{z \in [L]} p(z)=1$ implies 
\[
C (1+S(p_1;N)+S(p_2;K))=CZ = 1, \qquad Z = Z(p_1, p_2; N,K).
\]
Let us consider $\tilde{X} = X-n$ with its probability mass function 
\[
\tilde{p}(n)= p(n+N) = C p_1^{-n} \1_{[-N,0]}(n) + C p_2^{ n} \1_{[1,K]}(n).
\]
 Clearly $\var(X)=\var(\tilde{X})$. We have
\[
	\E \tilde{X}  = - C S_1(p_1;N)+C S_1(p_2;K), \qquad \E \tilde{X}^2 = CS_2(p_1; N) + CS_2(p_2;K).
\]
Therefore
\[
	\var(X) = \var(\tilde{X}) = \frac{S_2(x; N) + S_2(y;K)}{Z} - \frac{(S_1(x;N)- S_1(y;K))^2}{Z^2}.
\]
Our inequality is thus equivalent to
\[
	Z^4-Z^3-Z(S_2(p_1; N) + S_2(p_2;K))+ (S_1(p_1;N)- S_1(p_2;K))^2 \geq 0.
\]

\noindent \emph{Step 4.} Let us define
\[
P(x,y)=Z^4-Z^3-Z(S_2(x; N) + S_2(y;K))+ (S_1(x;N)- S_1(y;K))^2, \qquad Z=Z(x,y;N,K).
\]

We shall show that the polynomial inequality $P(x,y) \geq 0$ is satisfied for all $x,y \geq 0$ by showing that $P$ has nonnegative coefficients. We shall need the following lemmas.
\begin{lemma}\label{lem:sol-eq}
Let $S_0(x;N)=\sum_{n=0}^N x^n$. Let $s_l(k)$ denote the coefficient in front of $x^k$ in $S_0^l$.  Then 
\[
	s_l(k) = \left\{ \begin{array}{ll} {k+l-1 \choose l-1} & k \in [0,N]  \\
	{k+l-1 \choose l-1} - l {k+l-N-2 \choose l-1} & k \in [N+1,2N]
	 \end{array} \right. .
\]
\end{lemma}
\begin{proof}
The coefficient is the number of solutions to the equation $x_1+\ldots+x_l=k$ where $0 \leq x_i \leq N$, $i=1,\ldots,l$. If $k \leq N$ then the condition $x_i \leq N$ is superfluous and thus the number of solutions is well know and can be obtained by a standard application of the \emph{stars and bars} method. If $k>N$ then from this expression one needs to subtract the number of solutions with at least one number $x_i$ greater than $N$. If $k \leq 2N$, there is precisely one such number. We can choose its index $i$ in $l$ ways and define $\tilde{x_i}=x_i-N-1 \geq 0$. Then we are left with computing the number of solutions to $x_1+\ldots+ x_{i-1}+\tilde{x_i}+x_{i+1}+\ldots+ x_l = k-N-1$, where the variables are nonnegative.  Since $k-N-1 \leq N$, from the previous case we get that this equation has ${k+l-N-2 \choose l-1}$ solutions.   
\end{proof}

\begin{lemma}\label{lem:S1S2}
Let $S_i(x)=S_i(x;N)$, $i=0,1,2$ and let us write $S_0(x)S_2(x) - S_1(x)^2 = \sum_{k=0}^{2N} c_k x^k$.  Then
\[
	c_k = \left\{ \begin{array}{ll} {k+2 \choose 3} & k \in [0,N]  \\
	 {2N-k+2 \choose 3} & k \in [N+1,2N]
	 \end{array} \right. .
\]
\end{lemma}

\begin{proof}
We have 
\[
	S_0(x)S_2(x) - S_1(x)^2 = \frac12 \sum_{i,j=0}^N (i-j)^2 x^i x^j 
\]
and thus 
\[
c_k= \frac12 \sum_{\substack{i + j = k \\ 0 \leq i, j \leq N}} (i-j)^2 .
\]
If $k \in [0,N]$ then 
\[
c_k = \frac12 \sum_{i=0}^k (2i-k)^2 ={k+2 \choose 3},
\]
where the last equality can be checked by a direct computation.
If $k \in [N+1,2N]$ then
\[
	c_k = \frac12 \sum_{i=k-N}^N (2i-k)^2 =  \frac12 \sum_{i=0}^{2N-k} (2i -(2N-k))^2   ={2N-k+2 \choose 3},
\] 
by the previous equality.
\end{proof}

In order to prove that $P$ has nonnegative coefficients we write
$
	P(x, y) = \sum_{k = 0}^{4N}\sum_{l = 0}^{4K} p_{k, l}x^ky^l
$
and consider two cases.

\vspace{0.3cm}

\noindent \emph{Case 1.} Coefficients $p_{k,0}$ and $p_{0,k}$. The roles of $(x,N)$ and $(y,K)$ in $P(x,y)$ are symmetric, therefore it is enough to consider only $p_{k,0}$. Note that in this case we deal with coefficients of 
\[
P(x,0) = S_0^4(x)-S_0^3(x)- S_0(x) S_2(x) + S_1^2(x). 
\]
If $k>2N$ then the only contribution to $p_{k,0}$ comes from $S_0^4-S_0^3 = S_0^3(S_0-1)$, where both factors have nonnegative coefficients, so there is nothing to prove. If $k \in [0,N]$ then by Lemmas   \ref{lem:sol-eq} and \ref{lem:S1S2} we get
\[
	p_{k,0}= {k+3 \choose 3} - {k+2 \choose 2} - {k+2 \choose 3} = 0.
\]
Finally, if $k \in [N+1,2N]$ then we get
\begin{align*}
	p_{k,0} & = {k+3 \choose 3} - 4 {k -N+2 \choose 3} - {k+2 \choose 2} + 3 {k-N+1 \choose 2} - {2N-k+2 \choose 3 } \\
	&= \frac{1}{6} (k-N) \left(-2 k^2+4 k N-3 k+4 N^2+15 N+5\right),
\end{align*}
where again the last equality can be  verified directly. The first factor is positive and thus we have to show the positivity of the second factor, which is a concave quadratic function in $k$. Thus it is enough to check only $k=N, 2N$ in which case we get positive expressions $5 + 12 N + 6 N^2$ and $5 + 9 N + 4 N^2$. 

\vspace{0.3cm}

\noindent \emph{Case 2.} Coefficients $p_{k,l}$, where $k,l \geq 1$. It is enough to assume $1 \leq k \leq N$ and $1 \leq l \leq K$ since otherwise the only contribution comes from $Z^4-Z^3 = Z^3(Z-1)$, which has nonnegative  coefficients. In this case the coefficient in $Z(S_2(x;N)+S_2(y;K))$ is $k^2+l^2$, whereas the coefficient in $(S_1(x;N)-S_1(y;K))^2$ is $-2kl$, which together give contribution $-(k+l)^2$. Since $Z-1$ has nonnegative coefficients, the coefficients of $Z^3(Z-1)$ are at least as large as those of $(x+y)Z^3$. The coefficients of $Z^3=(1+S(x;N)+S(y,K))^3$ are at least as large as those of $3S(x;N)^2 S(y;K)+3S(x;N) S(y;K)^2$, which according to Lemma \ref{lem:sol-eq} are equal to $3 {k+1 \choose 2} + 3 {l+1 \choose 2}$. If suffices to observe that the coefficients of $(x+y)Z^3$ are therefore at least
\[
	3 \left( {k \choose 2} + {l+1 \choose 2} + {k+1 \choose 2} + {l \choose 2} \right) = 3(k^2+l^2) \geq (k+l)^2.
\]

\section{Proof of Theorem \ref{thm:poiss}}\label{sec:4}

\vspace{0.2cm}
\noindent \emph{Step 1.}
Let $n_0= \E X$ and let $\mu(n)$ be the probability mass function of $X$. Our goal is to prove the inequality $\mu(n_0) \geq e^{-n_0} \frac{n_0^{n_0}}{n_0!}$. By an approximation argument one can assume that $X$ has its  support contained in $[L]$. Note that $\mu(n)=\frac{1}{n!}p(n)$, where $p \in \mc{L}_{[L]}$.  We would like to maximize the linear (and thus convex) functional $\mu \to -\mu(n_0)$ under the constraints given by vectors $v_1=(\frac{1}{0!}, \frac{1}{1!}\ldots, \frac{1}{L!})$ (fixing $\mu$ to be a probability distribution) and $v_2=(0, \frac{1}{0!}\ldots, \frac{1}{(L-1)!})$, fixing the mean. Thus Lemma \ref{lem:convex} implies that the maximum is attained on sequences having at most two degrees of freedom and therefore for $\mu$ of the form $\mu(n)=c x_0^n \1_{[k,l]}(n)$ for some $0 \leq k \leq l$. As a consequence, in order to prove the inequality it is enough to consider only sequences $\mu$ of the form
\[
	\mu(n) = \frac{1}{f(x_0)} \cdot \frac{x_0^n}{n!} \1_{[k,l]}(n), \qquad \textrm{where} \quad f(x) = \sum_{i=k}^l \frac{x^i}{i!}. 
\]

\noindent \emph{Step 2.} One can assume that $f$ is non-constant. Clearly $n_0 = \E X=\sum_{i=k}^l i\cdot q(i)=x_0 \cdot \frac{f'(x_0)}{f(x_0)}$. Our goal is to prove the inequality
\[
	\frac{1}{f(x_0)} \cdot \frac{x_0^{n_0}}{n_0!}  \geq \frac{1}{e^{n_0}} \cdot \frac{n_0^{n_0}}{n_0! }.
\]
This simplifies to $f(x_0)\leq (\frac{ex_0}{n_0})^{n_0}$ which after taking the logarithm reads $\log f(x_0) \leq n_0(1+ \log(\frac{x_0}{n_0}))$. Recall that $n_0=x_0 \frac{f'(x_0)}{f(x_0)}$. Plugging this in gives the equivalent form 
\[
\log f(x_0) \leq x_0 \frac{f'(x_0)}{f(x_0)}\left(1-\log\left(\frac{f'(x_0)}{f(x_0)}\right) \right).
\]
It would therefore be enough to show that the function
\[
	h(x)=x\cdot \frac{f'(x)}{f(x) }-x\cdot \frac{f'(x)}{f(x) }\cdot \log\left(\frac{f'(x)}{f(x)}\right)-\log{f(x)}
\]
is nonnegative for all $x \geq 0$. Taking $x=x_0$ will then finish the proof.

\vspace{0.2cm}
\noindent \emph{Step 3.} By a direct computation we have 
\[
h'(x)=-\frac{\log{ \frac{f'(x)}{f(x)} }}{f^2(x)}\cdot (-x(f'(x))^2+xf(x)f''(x)+f(x)f'(x)).
\]

\noindent \emph{Claim 1.} For all $x\geq 0$ we have $-x(f'(x))^2+xf(x)f''(x)+f(x)f'(x)\geq 0$.

\begin{proof}[Proof of Claim 1.]
By Cauchy-Schwarz inequality
\begin{align*}
f(x)  (xf'(x)+x^2f''(x)) & =\left(\sum_{i=k}^l \frac{x^i}{i!}\right)\left(\sum_{i=k}^l \frac{x^i}{(i-1)!}+\frac{x^i}{(i-2)!}\right)= 
\left(\sum_{i=k}^l \frac{x^i}{i!} \right)\left(\sum_{i=k}^l \frac{ix^i}{(i-1)!}\right) \\ &  \geq \left(\sum_{i=k}^l \frac{x^i}{(i-1)!}\right)^2=(x f'(x))^2.
\end{align*}
The assertion follows by dividing both sides by $x$.
\end{proof}

\noindent \emph{Claim 2.} The function $\psi(x)= -\log{\frac{f'(x)}{f(x)}}$ has a unique zero $y_0 \in [0,\infty)$.

\begin{proof}
According to a theorem due to Gurvits \cite{G09-2} (see also \cite{G09, HNT21} for alternative proofs) a function of the form $\sum_{i=0}^\infty a_i \frac{x^i}{i!}$ is log-concave for $x \geq 0$ if the sequence $(a_i)_{i \geq 0}$ is log-concave. Thus $f(x)$ is log-concave for $x \geq 0$. Equivalently $\frac{f'(x)}{f(x)}$ is a decreasing function on $[0,\infty)$ and thus $\psi$ is increasing.

 If $f(0)=0$ then $\lim_{x\to 0^+}\frac{f'(x)}{f(x)}=\infty$ while $\lim_{x\to \infty}\frac{f'(x)}{f(x)}=0$. By intermediate value property $\psi$ has a unique zero in $[0,\infty)$. If $f(0)=f'(0)=1$ then  
$\psi(0)=0$ and $0$ is the unique zero of $\psi$.
\end{proof}

We can now easily finish the proof. By Claims 1 and 2 we see that $h'$ is nonpositive on $[0,y_0]$ and nonnegative on $[y_0,\infty)$. Therefore $h$ attains its minimum at $x=y_0$. It is therefore enough to check the inequality $h(y_0) \geq 0$. Clearly $\psi(y_0)=0$ implies that $f(y_0)=f'(y_0)$. Thus $h(y_0)=y_0-\log f(y_0)$. The inequality $h(y_0) \geq 0$ is therefore equivalent to $e^{y_0} \geq f(y_0)$ and is  obvious as $f(y_0)$ is a truncated sum defining the exponential function.

\subsection*{Acknowledgments} 
We would like to thank   Ioannis Kontoyiannis and Mokshay Madiman for useful discussions.

\end{document}